\documentclass[11pt]{amsart}

\def\R{{\mathbb {R}}}
\def\N{{\mathbb {N}}}

\def\A{{\mathcal {A}}}
\newcommand{\PP}{\mathcal{P}}
\newcommand{\ve}{\varepsilon}
\newcommand{\p}{\partial}
\def\eps{{\varepsilon}}
\def\dist{\operatorname{\mathrm{dist}}}

\newtheorem{teo}{Theorem}[section]
\newtheorem{lema}[teo]{Lemma}
\newtheorem{prop}[teo]{Proposition}
\newtheorem{corol}[teo]{Corollary}
\newcommand{\HH}{\mathcal{H}}
\theoremstyle{remark}
\newtheorem{remark}[teo]{Remark}

\theoremstyle{definition}
\newtheorem{defi}[teo]{Definition}

\numberwithin{equation}{section}

\begin{document}
\title[On the Sobolev trace Theorem for variable
exponent spaces in the critical range.]{On the Sobolev trace Theorem for variable
exponent spaces in the critical range.}

\author[J. Fern\'andez Bonder, N. Saintier and A. Silva]{Juli\'an Fern\'andez Bonder, Nicolas Saintier and Analia Silva}

\address[J. Fern\'andez Bonder and A. Silva]{IMAS - CONICET and Departamento de Matem\'atica, FCEyN - Universidad de Buenos Aires, Ciudad Universitaria, Pabell\'on I  (1428) Buenos Aires, Argentina.}

\address[N. Saintier]{Universidad Nacional General Sarmiento - Juan Mar\'ia Gutierrez 1150, Los Polvorines, Pcia de Bs. As., Argentina and Departamento de Matem\'atica, FCEyN - Universidad de Buenos Aires, Ciudad Universitaria, Pabell\'on I  (1428) Buenos Aires, Argentina.}

\email[J. Fernandez Bonder]{jfbonder@dm.uba.ar}

\urladdr[J. Fernandez Bonder]{http://mate.dm.uba.ar/~jfbonder}

\email[A. Silva]{asilva@dm.uba.ar}

\email[N. Saintier]{nsaintie@dm.uba.ar, nsaintie@ungs.edu.ar}

\urladdr[N. Saintier]{http://mate.dm.uba.ar/~nsaintie}



\begin{abstract}
In this paper we study the Sobolev Trace Theorem for variable exponent spaces with critical exponents. We find conditions on the best constant in order to guaranty the existence of extremals.
Then we give local conditions on the exponents and on the domain (in the spirit of Adimurthy and Yadava) in order to satisfy such conditions, and therefore to ensure the existence of extremals.
\end{abstract}

\maketitle

\section{Introduction}
The study of variable exponent Lebesgue and Sobolev spaces have deserved a great deal of attention in the last few years due to many interesting new applications including the mathematical modeling of electrorheological fluids (see \cite{Ru}) and image processing (see \cite{CLR}).
We refer to section 2 below for a brief account of the main rsults needed here, and to the book \cite{libro} for a complete account on these spaces.  

\medskip

One fundamental point in the study of these spaces is the generalization of the well--known Sobolev immersion Theorems. That is, if $\Omega\subset \R^N$ is a bounded domain and $p\colon\Omega\to [1,\infty)$ is a finite exponent such that $\sup_\Omega p<N$ the following immersions hold
$$
W^{1,p(x)}_0(\Omega)\hookrightarrow L^{q(x)}(\Omega)\qquad \text{and}\qquad W^{1,p(x)}(\Omega)\hookrightarrow L^{r(x)}(\partial\Omega),
$$
if the exponents $q\colon \Omega\to [1,\infty)$ and $r\colon\partial\Omega\to [1,\infty)$ verify the bounds
$$
q(x)\le p^*(x) := \frac{Np(x)}{N-p(x)}\qquad \text{and}\qquad r(x)\le p_*(x):= \frac{(N-1)p(x)}{N-p(x)}.
$$
These exponents $p^*(x)$ and $p_*(x)$ are called the critical Sobolev exponent and the critical Sobolev trace exponent respectively. (Some mild regularity assumptions on the exponents are needed in order for the immersions to hold, see \cite{libro} and Section 2).
These immersions can be restated as
\begin{align*}
&0<S(p(\cdot),q(\cdot),\Omega) := \inf_{v\in W^{1,p(x)}_0(\Omega)}\frac{\|\nabla v\|_{L^{p(x)}(\Omega)}}{\|v\|_{L^{q(x)}(\Omega)}}, \quad \text{and}\\
&0<T(p(\cdot),r(\cdot),\Omega) := \inf_{v\in W^{1,p(x)}(\Omega)}\frac{\|v\|_{W^{1,p(x)}(\Omega)}}{\|v\|_{L^{r(x)}(\Omega)}}.
\end{align*}
Here, the norms that are considered are the {\em Luxemburg norms}. We refer to Section 2 for the precise definitions.

\medskip

An important and interesting problem is the study of the existence of extremals for these immersions i.e. functions realizing the infimum in the definition of 
$S(p(\cdot),q(\cdot),\Omega)$ and $T(p(\cdot),r(\cdot),\Omega)$. 
When the exponents are {\em uniformly subcritical}, i.e.
$$ \inf_{\Omega}(p^*-q)>0 \qquad\text{and}\qquad \inf _{\partial\Omega}(p_*-r)>0, $$
the immersions are compact, and so the existence of extremals follows by a direct minimization procedure. The situation when the subcriticality is violated is much more complicated.

In constrast with the constant critical exponent case which has deserved a lot of attention since Aubin' seminal work \cite{Aubin}, the critical immersion for variable exponent have only been considered recently. In \cite{MOSS}, the authors study some cases where even if the subcriticality is violated, the immersion $W^{1,p(x)}_0(\Omega)\hookrightarrow L^{q(x)}(\Omega)$ remains compact. This result requires for very restrictive hypotheses on the exponents $p$ and $q$, so a more general result is desirable. In this direction, in \cite{FBSS1}, applying an extension of the P.L. Lions'  Concentration--Compactness Principle for the variable exponent case (see \cite{FBS1, Fu}) the authors proved that
$$ S(p(\cdot),q(\cdot),\Omega) \le \sup_{\ve>0}\inf_{x\in\A} S(p(\cdot),q(\cdot), B_\ve(x)), $$
where $\A=\{x\in \Omega\colon q(x) = p^*(x)\}$ is the critical set, and $B_\ve(x)$ is the ball centered at $x$ of radius $\ve$. 
Moreover, in that paper it is shown that if the strict inequality holds, namely 
$$ S(p(\cdot),q(\cdot),\Omega) < \sup_{\ve>0}\inf_{x\in\A} S(p(\cdot),q(\cdot), B_\ve(x)), $$
then there exists an extremal for $S(p(\cdot),q(\cdot),\Omega)$. Some conditions on $p$, $q$ and $\Omega$ are also given in order for this strict inequality to hold. We also refer to \cite{FBSS2} where this result is applied to obtain the existence of a solution to a critical equation involving the $p(x)-$Laplacian.

The purpose of this article is to extend the above mentioned results to the trace problem. That is, we assume hereafter that the subcriticality for the trace exponent fails in the sense that 
$$ \A_T :=\{x\in \partial\Omega\colon r(x) = p_*(x)\}\neq\emptyset, $$
and find conditions on the exponents $p$, $r$ and on the domain $\Omega$ in order to ensure the existence of an extremal for $T(p(\cdot),r(\cdot),\Omega)$. Up to our knowledge, this is the first paper where the critical trace inequality, in the context of variable exponent Sobolev spaces, is addressed.

\medskip

Concerning the constant exponent case, it is known, see \cite{FBR}, that\index{$\bar K(N,p)^{-1}$}
$$  T(p,p_*,\Omega)\le \bar K(N,p)^{-1} = \inf_{v\in \bar D^{1,p}(\R^N_+)} \frac{\|\nabla v\|_{L^p(\R^N_+)}}{\|v\|_{L^{p_*}(\R^{N-1})}}, $$
where $\bar D^{1,p}(\R^N_+)$ is the set of measurable functions $f(y,t)$ such that $\partial_i f\in L^p(\R^N_+)$, $i=1,\dots,N$, and $f(\cdot, 0) \in L^{p_*}(\R^{N-1})$. 
Moreover, in \cite{FBR} it is shown that if
\begin{equation}\label{t<k}
T(p,p_*,\Omega)<\bar K(N,p)^{-1},
\end{equation}
then there exists an extremal for the trace inequality.
Notice that one trivial global condition on $\Omega$ that implies \eqref{t<k} is
\begin{equation}\label{omegachiquito}
\frac{|\Omega|^{\frac{1}{p}}}{\HH^{N-1}(\partial\Omega)^{\frac{1}{p_*}}} < \bar K(N,p)^{-1},
\end{equation}
where $\HH^d$ denotes the $d-$dimensional Hausdorff measure. Observe that the family of sets verifying \eqref{omegachiquito} is large. Indeed for any fixed set $\Omega$, $\Omega_t:=t\cdot\Omega$ verifies \eqref{omegachiquito} for any $t>0$ small.

A more interesting and difficult task is to find local conditions on $\Omega$ ensuring \eqref{t<k}. For $p=2$ this was done by Adimurthy and Yadava in \cite{Adi} (see also Escobar \cite{Escobar} for a closely related problem) by using the fact that the extremals for $\bar K(N,2)^{-1}$ were explicitly known since the work of Escobar \cite{Escobar}. In fact, in \cite{Adi}, the authors proved that if the boundary of $\Omega$ contains a point with positive mean curvature, then \eqref{t<k} holds true.
Recently Nazaret \cite{Nazaret} found the extremals for $\bar K(N,p)^{-1}$ by means of mass transportation methods. These extremals are of the form \index{Trace Extremals}
$$ V_{\lambda, y_0}(y,t) = \lambda^{-\frac{N-p}{p-1}} V(\tfrac{y-y_0}{\lambda},\tfrac{t}{\lambda}),\qquad y\in \R^{N-1},\ t>0,  $$
with 
\begin{equation}\label{Extremal} 
 V(y,t) = r^{-\frac{N-p}{p-1}},\qquad r=\sqrt{(1+t)^2 + |y|^2}. 
\end{equation}
From the explicit knowledge of the extremals one can compute the value of the constant $\bar K(N,p)$ (see, for example, \cite{FBS}). It holds
$$
\bar K(N,p) = \pi^{\frac{1-p}{2}} \left(\frac{p-1}{N-p}\right)^{p-1} \left(\frac{\Gamma(\frac{p(N-1)}{2(p-1)})}{\Gamma(\frac{N-1}{2(p-1)})}\right)^{\frac{p-1}{N-1}},
$$
where $\Gamma(x) = \int_0^\infty t^{x-1}e^{-t}\, dt$ is the Gamma function.
Using these extremals, Fern\'andez Bonder and Saintier in \cite{FBS} extended \cite{Adi} by proving that \eqref{t<k} holds true if $\partial\Omega$ contains a point of positive mean curvature for $1<p<(N+1)/2$. See also \cite{Nazarov} for a related result. We also refer to \cite{Saintier2} where this question has been adressed in the case $p=1$.

\medskip

A slightly more general problem can be treated. Namely, consider $\Gamma\subset \partial\Omega$, $\Gamma\neq\partial\Omega$ a (possibly empty) closed set, and define
$$
W^{1,p(x)}_\Gamma = \overline{\{\phi\in C^\infty(\bar\Omega)\colon \phi \text{ vanishes in a neighborhood of }\Gamma\}},
$$
where the closure is taken in $\|\cdot\|_{W^{1,p(x)}(\Omega)}-$norm. This is the subspace of functions vanishing on $\Gamma$. Obviously, $W^{1,p(x)}_{\emptyset}(\Omega) = W^{1,p(x)}(\Omega)$. In general $W^{1,p(x)}_{\Gamma}(\Omega) = W^{1,p(x)}(\Omega)$ if and only if the $p(x)-$capacity of $\Gamma$ is $0$, see \cite{Harjuleto}. The main concern of this paper is the study of the existence problem of extremals for the best constant $T(p(\cdot), r(\cdot), \Omega, \Gamma)$ defined by
\begin{equation}\label{constante 2}
0<T(p(\cdot), r(\cdot), \Omega, \Gamma) := \inf_{v\in W^{1,p(x)}_\Gamma(\Omega)} \frac{\|v\|_{W^{1,p(x)}(\Omega)}}{\|v\|_{L^{r(x)}(\partial\Omega)}}.
\end{equation}
First, employing the same ideas as in \cite{MOSS} we obtain some restricted conditions on the exponents $p$ and $r$ guarantying that the immersion $W^{1,p(x)}(\Omega)\hookrightarrow L^{r(x)}(\partial\Omega)$ remains compact and so the existence of an extremal for $T(p(\cdot), r(\cdot), \Omega, \Gamma)$ holds true. As in the Sobolev immersion Theorem more general conditions for the existence of extremals are needed and these are the contents of our main results.

\medskip

In order to state our main results, we first need to introduce some notation.
The localized Sobolev trace constant $\bar T_{x}$ is defined, for $x\in \A_T$, as
\begin{equation}\label{constante 3}\index{$\bar T_{x}$}
\bar T_{x} = \sup_{\ve>0} T(p(\cdot), r(\cdot), \Omega_\ve, \Gamma_\ve) = \lim_{\ve\to 0} T(p(\cdot), r(\cdot), \Omega_\ve, \Gamma_\ve),
\end{equation}
where $\Omega_\ve = \Omega\cap B_\ve(x)$ and $\Gamma_\ve = \partial B_\ve(x) \cap \bar\Omega$. 
The smallest localized Sobolev trace constant is denoted by
\begin{equation}\label{barT}\index{$\bar T$}
\bar T := \inf_{x\in\A_T} \bar T_{x}.
\end{equation}

With these notations, our main results states that, under certain mild regularity assumptions on $p$ and $r$, the following inequalities hold true
$$ T(p(\cdot), r(\cdot), \Omega, \Gamma)\le \bar T \le \inf_{x\in \A_T} \bar K(N,p(x))^{-1}. $$
Moreover, if the following strict inequality holds
\begin{equation}\label{hip}
T(p(\cdot), r(\cdot), \Omega, \Gamma) <  \bar T,
\end{equation}
then there exists an extremal for \eqref{constante 2}.

So a natural main concern is to provide with conditions in order for \eqref{hip} to hold. We obtain, as in the constant exponent case, two types of conditions: local and global.

Global conditions are easier to obtain. In fact, it is fairly easy to see that if $\Omega$ is contracted enough then \eqref{hip} holds.

In order to find local conditions for \eqref{hip} to hold, a more refined analysis has to be made. The idea is to find a precise test function in order to estimate $T(p(\cdot), r(\cdot), \Omega, \Gamma)$. This test function is constructed by properly scaling and truncating the extremal for $\bar K(N,p(x))^{-1}$ around some point $x\in \A_T$. This estimate will give local conditions  ensuring that $T(p(\cdot), r(\cdot), \Omega, \Gamma) < \bar K(N,p(x))^{-1}$. 
The analysis is then completed by providing with conditions that ensure $\bar T_x = \bar K(N,p(x))^{-1}$, and requiring that $\bar T = \bar T_x$ for some $x\in \A_T$.

\medskip

\subsection*{Organization of the paper}
The rest of the paper is organized as follows. In Section 2, we collect some preliminaries on variable exponent spaces that will be used throughout the paper. In Section 3, by applying the method developed in \cite{MOSS}, we find conditions than ensure that the trace immersion remains compact although $\A_T\neq\emptyset$. As we mentioned in the introduction, these conditions are not satisfactory, so in the remaining of the paper we look for a general result that guaranty the existence of extremals. In Section 4 we revisit the proof of the Concentration--Compactness Theorem as stated in \cite{FBS1} to perform the corresponding adaptation for the trace inequality. In Section 5 we prove our main results, Theorem \ref{estimacionT} and Theorem \ref{condicionT} that provide with general conditions for the existence of extremals.  Finally, in Section 6 we provide both local and global conditions for the validity of $T(p(\cdot),r(\cdot),\Omega)<\bar T$.

\section{Preliminaries on variable exponent Sobolev spaces}

In this section we review some preliminary results regarding Lebesgue and Sobolev spaces with variable exponent. All of these results and a comprehensive study of these spaces can be found in \cite{libro}.

\medskip

The variable exponent Lebesgue space $L^{p(x)}(\Omega)$ is defined by
$$
L^{p(x)}(\Omega) = \Big\{u\in L^1_{\text{loc}}(\Omega) \colon \int_\Omega|u(x)|^{p(x)}\,dx<\infty\Big\}.
$$
This space is endowed with the norm
$$
\|u\|_{L^{p(x)}(\Omega)} = \|u\|_{p(x)} :=\inf\Big\{\lambda>0:\int_\Omega\Big|\frac{u(x)}{\lambda}\Big|^{p(x)}\,dx\leq 1\Big\}
$$
 We can define the variable exponent Sobolev space $W^{1,p(x)}(\Omega)$ by
$$
W^{1,p(x)}(\Omega) = \{u\in L^{p(x)}(\Omega) \colon \partial_i u\in L^{p(x)}(\Omega) \text{ for } i=1,\dots,N\},
$$
where $\partial_i u = \frac{\partial u}{\partial x_i}$ is the $i^{th}-$distributional partial derivative of $u$.
This space has a corresponding modular given by \index{$\rho_{1,p(x)}$}
$$
\rho_{1,p(x)}(u) := \int_\Omega |u|^{p(x)} + |\nabla u|^{p(x)}\, dx
$$
which yields the norm 
$$
\|u\|_{W^{1,p(x)}(\Omega)} = \| u\|_{1,p(x)} : = \inf\Big\{\lambda>0\colon \rho_{1, p(x)}\left(\frac{u}{\lambda}\right)\leq 1\Big\}.
$$
Another possible choice of norm in $W^{1,p(x)}(\Omega)$ is  $\|u\|_{p(x)} + \| \nabla u \|_{p(x)}$. Both norms turn out to be equivalent but we use the first one for convenience.

The following result is proved in \cite{Fan, KR} (see also \cite{libro}, pp. 79, Lemma 3.2.20 (3.2.23)).

\begin{prop}[H\"older-type inequality]\label{Holder}
Let $f\in L^{p(x)}(\Omega)$ and $g\in L^{q(x)}(\Omega)$. Then the following inequality holds
$$
\| f(x)g(x) \|_{L^{s(x)}(\Omega)}\le  \Big( \Big(\frac{s}{p}\Big)^+ + \Big(\frac{s}{q}\Big)^+\Big) \|f\|_{L^{p(x)}(\Omega)}\|g\|_{L^{q(x)}(\Omega)},
$$
where
$$
\frac{1}{s(x)} = \frac{1}{p(x)} + \frac{1}{q(x)}.
$$
\end{prop}

From now on, we define the classes of exponents that we deal with. Let $\PP(\Omega)$ be the set of Lebesgue measurable functions $p\colon \Omega\to [1,\infty)$ and let $\PP(\partial\Omega)$ be the set of $\HH^{N-1}-$measurable functions $r\colon\partial\Omega\to [1,\infty)$.

In order to state the trace Theorem we need to define the Lebesgue spaces on $\partial \Omega$. We assume that $\Omega$ is $C^1$ so $\partial \Omega$ is a $(N-1)-$dimensional $C^1$ immersed manifold on $\R^N$ (less regularity on $\partial \Omega$ is enough for the trace Theorem to hold, but the $C^1$ regularity is enough for our purposes). Therefore the boundary measure agrees with the $(N-1)-$Hausdorff measure restricted to $\partial\Omega$. We denote this measure by $dS$. Then, the Lebesgue spaces on $\partial\Omega$ are defined as
$$
L^{r(x)}(\partial\Omega):= \Big\{ u\in L^1_{\text{loc}}(\partial\Omega, dS)\colon \int_{\partial\Omega} |u(x)|^{r(x)}\, dS<\infty\Big\}
$$
and the corresponding (Luxemburg) norm is given by
$$
\|u\|_{L^{r(x)}(\partial\Omega)} = \|u\|_{r(x), \partial\Omega} := \inf\Big\{\lambda>0\colon \int_{\partial\Omega} \Big|\frac{u(x)}{\lambda}\Big|^{r(x)}\, dS\le 1\Big\}.
$$

Throughout this paper the following notation will be used: For a $\mu-$measurable function $f$ we denote $f^+ := \sup f$ and $f^- := \inf f$, where by $\sup$ and $\inf$ we denote the essential supremum and essential infimum respectively with respect to the measure $\mu$.

The Sobolev trace Theorem is proved in \cite{Fan}. When the exponent is critical, it requires more regularity on the exponent $p(x)$ (Lipschitz regularity is enough). This regularity can be relaxed when the exponent is strictly subcritical. It holds,

\begin{teo}\label{trace}
Let $\Omega\subseteq\R^N$ be an open bounded domain with Lipschitz boundary and let $p\in \PP(\Omega)$ be such that $p\in W^{1,\gamma}(\Omega)$ with $1\leq p_{-}\leq p^{+}<N<\gamma$. Then there is a continuous boundary trace embedding $W^{1,p(x)}(\Omega)\subset L^{p_*(x)}(\partial\Omega)$.
\end{teo}

\begin{teo}
Let $\Omega\subset\R^N$ be an open bounded domain with Lipschitz boundary. Suppose that $p\in C^0(\bar{\Omega})$ and $1<p^-\leq p^+<N$. If $r\in \PP(\partial\Omega)$ is uniformly subcritical then the boundary trace embedding $W^{1,p(x)}(\Omega)\to L^{r(x)}(\partial\Omega)$ is compact.
\end{teo}

\begin{corol}
Let $\Omega\subset\R^N$ be an open bounded domain with Lipschitz boundary. Suppose that $p\in C^0(\bar{\Omega})$ and $1<p_{-}\leq p_{+}<N$. If $r\in C^0(\partial\Omega)$ satifies the condition
$$
1\leq r(x)<p_*(x)\quad x\in\partial\Omega
$$
then there is a compact boundary trace embedding $W^{1,p(x)}(\Omega)\to L^{r(x)}(\partial\Omega)$
\end{corol}

The following proposition, also proved in \cite{KR}, will be most useful (see also \cite{libro}, Chapter 2, Section 1).
\begin{prop}\label{norma.y.rho}
Set $\rho(u):=\int_\Omega|u(x)|^{p(x)}\,dx$. For $u\in L^{p(x)}(\Omega)$ and $\{u_k\}_{k\in\N}\subset L^{p(x)}(\Omega)$, we have
\begin{align}
& u\neq 0 \Rightarrow \Big(\|u\|_{L^{p(x)}(\Omega)} = \lambda \Leftrightarrow \rho(\frac{u}{\lambda})=1\Big).\\
& \|u\|_{L^{p(x)}(\Omega)}<1 (=1; >1) \Leftrightarrow \rho(u)<1(=1;>1).\\
& \|u\|_{L^{p(x)}(\Omega)}>1 \Rightarrow \|u\|^{p^-}_{L^{p(x)}(\Omega)} \leq \rho(u) \leq \|u\|^{p^+}_{L^{p(x)}(\Omega)}.\\
& \|u\|_{L^{p(x)}(\Omega)}<1 \Rightarrow \|u\|^{p^+}_{L^{p(x)}(\Omega)} \leq \rho(u) \leq \|u\|^{p^-}_{L^{p(x)}(\Omega)}.\\
& \lim_{k\to\infty}\|u_k\|_{L^{p(x)}(\Omega)} = 0 \Leftrightarrow \lim_{k\to\infty}\rho(u_k)=0.\\
& \lim_{k\to\infty}\|u_k\|_{L^{p(x)}(\Omega)} = \infty \Leftrightarrow \lim_{k\to\infty}\rho(u_k) = \infty.
\end{align}
\end{prop}

For much more on these spaces, we refer to \cite{libro}.

\section{Compact case}

 In this section we find conditions on the exponents $p\in \PP(\Omega)$ and $r\in \PP(\partial\Omega)$ that imply that the immersion $W^{1,p(x)}(\Omega)\hookrightarrow L^{r(x)}(\partial\Omega)$ remains compact. Therefore, in this case, the existence of extremals follows directly by minimization.

Roughly speaking, these conditions require the critical set to be {\em small}, and also a strict control on how the exponent $r$ reaches the critical one when one is approaching the critical set $\A_T$. For the Sobolev immersion $W^{1,p(x)}_0(\Omega)\hookrightarrow L^{q(x)}(\Omega)$, this result was obtained in \cite{MOSS}. Following the same ideas we can prove a similar result for the trace immersion.

First, we define the upper Minkowsky content for sets contained in $\p\Omega$. We say that a compact set $K\subset \p\Omega$ has finite $(N-1-s)-$boundary dimensional upper Minkowsky content if there exists a constant $C>0$ such that\index{boundary upper Minkowsky content}
$$ \HH^{N-1}(K(r)\cap\p\Omega) \le Cr^s, \qquad\text{for all } r>0, $$
where $K(r)=\{x\in\R^N\colon \dist(x,K)<r\}$. The result is the following: 

\begin{teo}\label{teo MOSS T}
Let $\varphi\colon [r_0^{-1},\infty)\to(0,\infty)$ be a continuous function such that: $\varphi(r)/\ln r$ is nonincreasing in $[r_0^{-1},\infty)$ for some $r_0\in(0,e^{-1})$ and $\varphi(r)\to\infty$ as $r\to\infty$. Let $K\subset \p\Omega$ be a compact set whose $(N-1-s)-$boundary dimensional upper Minkowski content is finite for some $s$ with $0<s\leq N-1$.

Let $p\in \PP(\Omega)$ and  $r\in \PP(\partial\Omega)$ be such that $p^+<N$ and $r(x)\le p_*(x)$. Assume that $r(x)$ is subcritical outside a neighborhood of $K$, i.e.  $\inf_{\partial\Omega\setminus K(r_0)} (p_*(x)-r(x))>0$. Moreover, assume that $r(x)$ reaches $p_*(x)$ in $K$ at the following rate
$$
r(x)\leq p_*(x)-\frac{\varphi(\frac{1}{\dist(x,K)})}{\ln(\frac{1}{\dist(x,K)})} \quad \text{for almost every } x\in K(r_0)\cap\partial\Omega.
$$
Then the embedding $W^{1,p(x)}(\Omega)\hookrightarrow L^{r(x)}(\partial\Omega)$ is compact.
\end{teo}
\begin{proof}
Let us prove that
\begin{equation}\label{0}
\lim_{\varepsilon\to 0^+}\sup\Big\{\int_{K(\varepsilon)\cap\partial\Omega} |v(x)|^{r(x)}\,dS \colon v\in W^{1,p(x)}(\Omega)\text{ and }\|v\|_{W^{1,p(x)}(\Omega)}\leq 1 \Big\}=0.
\end{equation}

First, we take $\beta$ such that  $0<\beta< s/p^+_*$ and $\varepsilon>0$ such that $\varepsilon^{-1}>r_0^{-1}$ and $\varphi(\frac{1}{\varepsilon})\geq 1$. For each $n\in\N$ we consider $\eta_n=\varepsilon^{-\beta n}$. We choose $x\in(K(\varepsilon^n)\setminus K(\varepsilon^{n+1}))\cap\partial\Omega$, then , we have
$$
\eta_n^{r(x)-p_*(x)}
\leq \eta_n^{-\frac{\varphi\left(\frac{1}{\dist(x,K)}\right)}{\ln\left(\frac{1}{\dist(x,K)}\right)}}
\leq \eta_n^{-\frac{\varphi\left(\frac{1}{\varepsilon^{n+1}}\right)}{\ln\left(\frac{1}{\varepsilon^{n+1}}\right)}}
= \varepsilon^{-\frac{\beta n}{n+1}\varphi\left(\frac{1}{\varepsilon^{n+1}}\right)}
= A_n. 
$$
On the other hand, we know that $\HH(K(r)\cap\partial\Omega)\leq C r^s$ and we can estimate the following term
$$
\int_{(K(\varepsilon^n)\setminus K(\varepsilon^{n+1}))\cap\partial\Omega} \eta_n^{r(x)}\,dS\leq\eta_n^{p^+_*}\int_{K(\varepsilon^n)\cap\partial\Omega}\,dS\leq C\varepsilon^{n(s-\beta p^+_*)}
$$
Now, we have
\begin{align*}
&\int_{(K(\varepsilon^n)\setminus K(\varepsilon^{n+1}))\cap\partial\Omega}|v(x)|^{r(x)}\,dS\\
&\leq\int_{(K(\varepsilon^n)\setminus K(\varepsilon^{n+1}))\cap\partial\Omega}|v(x)|^{r(x)}\left(\frac{|v(x)|}{\eta_n}\right)^{p_*(x)-r(x)}\,dS+
\int_{(K(\varepsilon^n)\setminus K(\varepsilon^{n+1}))\cap\partial\Omega} \eta_n^{r(x)}\,dS\\
&\leq A_n \int_{(K(\varepsilon^n)\setminus K(\varepsilon^{n+1}))\cap\partial\Omega}|v(x)|^{p_*(x)}\,dS+C\varepsilon^{n(s-\beta p^+_*)}
\end{align*}
for each $n_0\in\N$, we obtain
\begin{align*}
\int_{K(\varepsilon^{n_0})\cap\partial\Omega}|v(x)|^{r(x)}\,dS&=\sum^\infty_{n=n_0}\int_{(K(\varepsilon^n)\setminus K(\varepsilon^{n+1}))\cap\partial\Omega}|v(x)|^{r(x)}\,dS\\
&\leq (\sup_ {n\ge n_0} A_n) \int_{K(\varepsilon^{n_0})\cap\partial\Omega}|v(x)|^{p_*(x)}\,dS+C\sum^\infty_{n=n_0}\varepsilon^{n(s-\beta p^+_*)}
\end{align*}
Using that $\|v\|_{p_*,\partial\Omega}\leq C\|v\|_{1,p}$ and that $(s-\beta p^+_*)>0$, we can conclude \eqref{0}.

Finally, let $\{v_n\}_{n\in \N}\subset W^{1,p(x)}(\Omega)$ and $v\in W^{1,p(x)}(\Omega)$ be such that
$$
v_n\rightharpoonup v \quad \text{weakly in } W^{1,p(x)}(\Omega).
$$
Then,
\begin{align*}
&v_n\rightharpoonup v \quad \text{weakly in } L^{r(x)}(\partial\Omega),\\
&v_n\to v \quad \text{strongly  in } L^{s(x)}(\partial\Omega) \text{ for every } s \text{ such that } \inf_{\partial\Omega}(p_*(x)-s(x)) >0,
\end{align*}
therefore $v_n\to v$ in $L^{r(x)}(\partial\Omega\setminus K(\varepsilon))$ for each $\ve>0$ small. Hence,
\begin{align*}
\limsup_{n\to \infty}\int_{\partial\Omega}|v_{n}(x)-v(x)|^{r(x)}\, dS =& \limsup_{n\to \infty}\Big(\int_{K(\varepsilon)\cap\partial\Omega}|v_{n}(x)-v(x)|^{r(x)}\,dS\\
&+\int_{\partial\Omega\setminus K(\varepsilon)}|v_{n}(x)-v(x)|^{r(x)}\,dS\Big)\\
\le & \sup_{n\in \N} \int_{K(\varepsilon)\cap\partial\Omega}|v_{n}(x)-v(x)|^{r(x)}\,dS
\end{align*}
So, by \eqref{0}, we conclude the desired result.
\end{proof}

Now it is straightforward to derive, analogous to Corollary 3.5 in \cite{MOSS},

\begin{corol}
Let $p\in \PP(\Omega)$ be such that $p^+<N$ and let $r\in \PP(\partial\Omega)$. Suppose that there exist $x_0\in\Omega$, $C>0$, $n\in\N$, $r_0>0$ such that $\inf_{\partial\Omega\setminus B_{r_0}(x_0)}(p_*(x)-r(x))>0$ and $r(x)\leq p_*(x)-c\frac{\ln^n(\frac{1}{|x-x_0|})}{\ln(\frac{1}{|x-x_0|})}$ for almost every $x\in \partial\Omega\cap B_{r_0}(x_0)$. Then the embedding $W^{1,p(x)}(\Omega)\hookrightarrow L^{r(x)}(\partial\Omega)$  is compact.
\end{corol}

\section{The concentration--compactness principle for the Sobolev trace immersion}\label{section.ccp-traza}

This section is devoted to the extension of the CCP to the trace immersion.

Let $r\in \PP(\partial\Omega)$ be a continuous critical exponent in the sense that
$$
\A_T:=\{x\in \p\Omega\colon r(x)=p_*(x)\}\neq\emptyset.
$$
We define the Sobolev trace constant in $W^{1,p(x)}_{\Gamma}(\Omega)$ as
$$
T(p(\cdot), r(\cdot), \Omega, \Gamma) := \inf_{v\in W^{1,p(x)}_{\Gamma}(\Omega) } \frac{\|v\|_{1,p(x)}}{\|v\|_{r(x), \partial\Omega}} = \inf_{v\in W^{1,p(x)}_{\Gamma}(\Omega) } \frac{\|v\|_{1,p(x)}}{\|v\|_{r(x), \partial\Omega\setminus\Gamma}}
$$

More precisely, we prove
\begin{teo}\label{CCT}\index{The concentration--compactness principle for the Sobolev trace immersion}
Let $\{u_n\}_{n\in\N}\subset W^{1,p(x)}(\Omega)$ be a sequence such that $u_n\rightharpoonup u$ weakly in $W^{1,p(x)}(\Omega)$. Then there exists a countable set $I$, positive numbers $\{\mu_i\}_{i\in I}$ and $\{\nu_i\}_{i\in I}$ and points $\{x_i\}_{i\in I}\subset \A_T\subset\partial\Omega$ such that
\begin{align}
\label{nu}& |u_n|^{r(x)} \,dS \rightharpoonup \nu = |u|^{r(x)}\,dS + \sum_{i\in I} \nu_i \delta_{x_i} \qquad \text{weakly-* in the sense of measures,}\\
\label{mu}& |\nabla u_n|^{p(x)}\, dx \rightharpoonup \mu \geq |\nabla u|^{p(x)}\, dx + \sum_{i\in I} \mu_i \delta_{x_i} \qquad \text{weakly-* in the sense of measures,}\\
\label{refinement}& \bar T_{x_i} \nu_i^{\frac{1}{r(x_i)}}\leq \mu_i^{\frac{1}{p(x_i)}},
\end{align}\index{$\bar T_{x_i}$}
where $\bar T_{x_i} = \sup_{\varepsilon>0} T(p(\cdot), q(\cdot), \Omega_{\ve, i}, \Gamma_{\ve, i})$ is the localized Sobolev trace constant where
$$
\Omega_{\ve, i} = \Omega\cap B_{\ve}(x_i) \quad \text{and}\quad \Gamma_{\ve, i} := \partial B_\ve(x_i)\cap \Omega.
$$
\end{teo}

\begin{proof}
The proof is very similar to the one for the Sobolev immersion Theorem, see \cite{FBSS1}, so we only make a sketch stressing the differences between the two cases.

As in \cite[Theorem 1.1]{FBS1} it is enough to consider the case where $u_n\rightharpoonup 0$ weakly in $W^{1,p(x)}(\Omega)$.

\medskip
Take $\phi\in C^\infty(\bar{\Omega})$. According to Theorem \ref{trace} we have
\begin{equation}\label{T1}
T(p(\cdot),q(\cdot),\Omega) \|\phi u_j\|_{r(x)}  \leq \|\phi u_j\|_{1,p(x)}.
\end{equation}
We have that
$$
\|\phi u_j\|_{1,p(x)}\leq C(\|\nabla(\phi u_j)\|_{p(x)}+\|\phi u_j\|_{p(x)})
$$
On the other hand,
$$
|\ \|\nabla(\phi u_j)\|_{p(x)} - \|\phi\nabla u_j\|_{p(x)} |\le \| u_j \nabla \phi\|_{p(x)}.
$$
Then, as $u_n\rightharpoonup 0$, we observe that the right hand side of the inequality converges to 0. In fact, we can assume that  $\rho_{p(x)}(u)< 1$, then
\begin{align*}
\| u_j \nabla \phi\|_{p(x)}&\le (\|\nabla \phi\|_{\infty}+1)^{p^+} \| u_j \|_{p(x)}\\
& \le (\|\nabla \phi\|_{\infty}+1)^{p^+} \rho_{p(x)}(u_j)^{1/p_-}\to 0
\end{align*}
We the same argument, we obtain that
$$
\|\phi u_j\|_{p(x)}\to 0
$$
Finally, if we take the limit for $j\to \infty$ in \eqref{T1}, we arrive at
\begin{equation}\label{RHT}
T(p(\cdot),r(\cdot),\Omega) \|\phi\|_{L_\nu^{r(x)}(\partial\Omega)} \leq\|\phi\|_{L_\mu^{p(x)}(\Omega)},
\end{equation}
for every $\phi\in C^\infty(\bar\Omega)$. Observe that if $\phi\in C^\infty_c(\R^N)$ and $U\subset \R^N$ is any open set containing the support of $\phi$, the constant in \eqref{RHT} can be replaced by $T(p(\cdot),q(\cdot),\Omega\cap U, \partial U\cap\Omega)$.

Now, the exact same proof of \cite[Theorem 1.1]{FBS1} implies that the points $\{x_i\}_{i\in I}$ must belong to the critical set $\A_T$.

Let $\phi\in C^\infty_c(\R^N)$ be such that $0\leq\phi\leq1$, $\phi(0)=1$ and supp$(\phi)\subset B_1(0)$. Now, for each $i\in I$ and $\varepsilon>0$, we denote $\phi_{\varepsilon,i}(x):= \phi((x-x_i)/\varepsilon)$.

From \eqref{RHT} and the subsequent remark we obtain
$$
T(p(\cdot), r(\cdot), \Omega_{\ve,i}, \Gamma_{\ve, i}) \|\phi_{\varepsilon,i}\|_{L^{r(x)}_\nu(\partial\Omega \cap B_\varepsilon(x_i))} \leq \|\phi_{\varepsilon,i}\|_{L^{p(x)}_\mu(\Omega \cap B_\varepsilon(x_i))}.
$$

As in \cite{FBS1}, we have
\begin{align*}
\rho_\nu(\phi_{i_0,\varepsilon}) &:= \int_{\partial\Omega \cap B_\varepsilon(x_{i_0})}|\phi_{i_0,\varepsilon}|^{r(x)}\,d\nu \\
&= \int_{\partial\Omega \cap B_\varepsilon(x_{i_0})} |\phi_{i_0,\varepsilon}|^{r(x)}|u|^{r(x)}\, dS + \sum_{i\in I} \nu_i\phi_{i_0,\varepsilon}(x_i)^{r(x_i)}\\
&\geq \nu_{i_0}.
\end{align*}

From now on, we will denote
\begin{align*}
r^+_{i,\varepsilon}:=\sup_{\partial\Omega \cap B_\varepsilon(x_i)}r(x),\qquad r^-_{i,\varepsilon}:=\inf_{\partial\Omega \cap B_\varepsilon(x_i)}r(x),\\
p^+_{i,\varepsilon}:=\sup_{\Omega \cap B_\varepsilon(x_i)}p(x),\qquad p^-_{i,\varepsilon}:=\inf_{\Omega \cap B_\varepsilon(x_i)}p(x).
\end{align*}

If $\rho_\nu(\phi_{i_0,\varepsilon})<1$ then
$$
 \|\phi_{i_0,\varepsilon}\|_{L^{r(x)}_\nu (\partial \Omega\cap B_\varepsilon(x_{i_0}))} \ge \rho_\nu(\phi_{i_0,\varepsilon})^{1/r^-_{i,\varepsilon}} \ge \nu_{i_0}^{1/r^-_{i,\varepsilon}}.
$$
Analogously, if $\rho_\nu(\phi_{i_0,\varepsilon})>1$ then
$$
\|\phi_{i_0,\varepsilon}\|_{L^{r(x)}_\nu(\partial\Omega\cap B_\varepsilon(x_{i_0}))} \ge \nu_{i_0}^{1/r^+_{i,\varepsilon}}.
$$
Therefore,
$$
T(p(\cdot), r(\cdot), \Omega_{\ve,i}, \Gamma_{\ve, i}) \min\Big\{\nu_i^\frac{1}{r^+_{i,\varepsilon}}, \nu_i^\frac{1}{r^-_{i,\varepsilon}}\Big\}  \leq \|\phi_{i,\varepsilon}\|_{L^{p(x)}_\mu(\Omega \cap B_\varepsilon(x_i))} .
$$

On the other hand,
$$
\int_{\Omega \cap B_\varepsilon(x_i)} |\phi_{i,\varepsilon}|^{p(x)}\,d\mu\leq \mu(\Omega \cap B_{\varepsilon}(x_i))
$$
hence
\begin{align*}
 \|\phi_{i,\varepsilon}\|_{L^{p(x)}(\Omega \cap B_\varepsilon(x_i))} &\leq  \max \Big\{ \rho_\mu (\phi_{i,\varepsilon})^\frac{1}{p^+_{i,\varepsilon}}, \rho_\mu(\phi_{i,\varepsilon})^\frac{1} {p^-_{i,\varepsilon}}\Big\}\\
&\le \max \Big\{ \mu(\Omega \cap B_\varepsilon(x_i))^\frac{1}{p^+_{i,\varepsilon}}, \mu(\Omega \cap B_\varepsilon(x_i))^\frac{1}{p^-_{i,\varepsilon}}\Big\},
\end{align*}
so we obtain,
$$
T(p(\cdot), r(\cdot), \Omega_{\ve,i}, \Gamma_{\ve, i}) \min \Big\{ \nu_i^\frac{1}{r^+_{i,\varepsilon}}, \nu_i^\frac{1}{r^-_{i,\varepsilon}}\Big\} \leq \max \Big\{ \mu(\Omega \cap B_\varepsilon(x_i))^\frac{1}{p^+_{i,\varepsilon}}, \mu(\Omega \cap B_\varepsilon(x_i))^\frac{1}{p^-_{i,\varepsilon}}\Big\}.
$$
As $p$ and $r$ are continuous functions and as $r(x_i) = p_*(x_i)$, letting $\varepsilon\to 0$, we get
$$
\bar T_{x_i} \nu_i^{1/p_*(x_i)} \le \mu_i^{1/p(x_i)},
$$
where $\mu_i := \lim_{\varepsilon\to 0}\mu(\Omega\cap B_\varepsilon(x_i))$.

The proof is now complete.
\end{proof}

\section{Non-compact case}

In this section we parallel the results for the Sobolev immersion Theorem obtained in \cite{FBSS1}, to the Sobolev trace Theorem.

In that spirit, the result we obtain states that if the Sobolev trace constant is strictly smaller that the smallest localized Sobolev trace constant in the critical set $\A_T$, then there exists an extremal for the trace inequality.

Then, the objective will be to find conditions on $p(x), r(x)$ and $\Omega$ in order to ensure that strict inequality. We find global and local conditions.

As in \cite{FBSS1}, global conditions are easily obtained and they say that if the surface measure of the boundary is larger than the volume of the domain, then the strict inequality holds and therefore an extremal for the trace inequality exists.

Once again, local conditions are more difficult to find. In this case, the geometry of the domain comes into play.

We begin with a lemma that gives a bound for the constant $T(p(\cdot), r(\cdot), \Omega, \Gamma)$.
\begin{lema}\label{desigualdad.primera.T}
Assume that the exponents $p\in \PP(\Omega)$ and $r\in \PP(\partial\Omega)$ are continuous functions with modulus of continuity $\rho$ such that
$$
\ln(\lambda) \rho(\lambda)\to 0\quad \text{as }\lambda\to 0+.
$$
Then, it holds that
$$
T(p(\cdot), r(\cdot), \Omega, \Gamma)\leq \inf_{x\in \A_T}\bar{K}(N,p(x))^{-1}.
$$
\end{lema}
\begin{proof}
The proof uses the same rescaling argument as in \cite{FBSS1} but we have to be more careful with the boundary term.

Let $ x_0\in\A_T$. Without loss of generality, we can assume that $x_0=0$ and that there exists $r>0$ such that
$$
\Omega_r := B_r\cap\Omega=\{x \in B_r\colon  x_N>\psi(x')\}, \qquad B_r\cap\partial\Omega=\{x\in B_r\colon  x_N=\psi(x')\},
$$
where $x=(x',x_N)$, $x'\in\R^{N-1}$, $x_N\in \R$, $B_r$ is the ball centered at the origin of radius $r$ and $\psi\colon \R^{N-1}\to \R$ is of class $C^2$ with $\psi(0)=0$ and $\nabla \psi(0)=0$.

First, we observe that our regularity assumptions on $p$ and $r$ imply that
\begin{align*}
& r(\lambda x) = r(0) + \rho_1(\lambda, x) = p_*(0) + \rho_1(\lambda, x),\\
& p(\lambda x) = p(0) + \rho_2(\lambda, x),
\end{align*}
with $\lim_{\lambda\to 0+}\lambda^{\rho_k(\lambda, x)} = 1$ uniformly in $\bar \Omega_r$. From now on, for simplicity, we write $p=p(0)$ and $p_*=p_*(0)=r(0)$.

Now, let $\phi\in C^\infty_c(\R^N)$, and define $\phi_\lambda$ to be the rescaled function around $0\in \A_T$ as $\phi_\lambda=\lambda^\frac{-(N-1)}{p_*}\phi(\frac{x}{\lambda})$ and observe that, since $\Gamma$ is closed and $0\not\in\Gamma$, $\phi_\lambda\in W^{1,p(x)}_\Gamma(\Omega)$ for $\lambda$ small. Then we have
\begin{align*}
\int_{\partial\Omega} \phi_\lambda(x)^{r(x)}\, dS &=\int_{\partial\Omega_\lambda} \lambda^{\frac{-(N-1)\rho_1(\lambda,y)}{p_*}} \phi(y)^{p_* + \rho_1(\lambda, y)}\, dS\\
 &=\int_{\R^{N-1}}\lambda^{\frac{-(N-1)\rho_1(\lambda,y)}{p_*}} \phi(y',\psi_\lambda(y'))^{p_*+\rho_1(\lambda, y')}\sqrt{1+|\nabla \psi_\lambda(y')|^2}\,dy',
\end{align*}
where $\Omega_\lambda=\frac{1}{\lambda}\cdot\Omega$ and $\psi_\lambda(y') = \tfrac{1}{\lambda}\psi(\lambda y')$.

Since $\psi(0)=0$ and $\nabla\psi(0)=0$ we have that $\psi_\lambda(y') = O(\lambda)$ and $|\nabla \psi_\lambda(y')| = O(\lambda)$ uniformly in $y'$ for $y'\in \text{supp}(\phi)$ which is compact. Moreover, our assumption on $\rho_1$ imply that
$$
\lambda^{\frac{-(N-1) \rho_1(\lambda, y)}{p_*}}\phi(y)^{\rho_1(\lambda, y)}\to 1 \mbox{ when }\lambda\to 0+
$$
uniformly in $y$.

Therefore, we get
\begin{equation}\label{eq41}
\rho_{r(x),\partial\Omega}(\phi_\lambda) = \int_{\partial\Omega} \phi_\lambda(x)^{r(x)}\, dS  \to \int_{\R^{N-1}}|\phi(y',0)|^{p_*}\,dy', \qquad \text{as } \lambda\to 0+.
\end{equation}

In particular, \eqref{eq41} imply that $\|\phi_\lambda\|_{r(x), \partial\Omega}$ is bounded away from 0 and $\infty$. Moreover,  arguing as before, we find
\begin{align*}
1 &= \int_{\partial\Omega} \left(\frac{\phi_\lambda(x)}{\|\phi_\lambda\|_{r(x), \partial\Omega}}\right)^{r(x)}\, dS\\  &=\int_{\R^{N-1}}\lambda^{\frac{-(N-1)\rho_1(\lambda,y)}{p_*}} \left(\frac{\phi(y',\psi_\lambda(y'))}{\|\phi_\lambda\|_{r(x), \partial\Omega}}\right)^{p_*+\rho_1(\lambda, y')}\sqrt{1+|\nabla \psi_\lambda(y')|^2}\,dy',
\end{align*}
so
$$
\lim_{\lambda\to 0+}\|\phi_\lambda\|_{r(x), \partial\Omega} = \|\phi\|_{p_*, \partial\R^N_+}.
$$

For the gradient term, we have
\begin{align*}
\int_{\Omega} |\nabla\phi_\lambda|^{p(x)}\, dx &=  \int_{\Omega} \lambda^{-\frac{N}{p}p(x)} |\nabla\phi(\tfrac{x}{\lambda})|^{p(x)}\, dx\\
&= \int_{\Omega_\lambda} \lambda^{-\frac{N}{p}\rho_2(\lambda, y)} |\nabla\phi(y)|^{p + \rho_2(\lambda, y)}\,dy.
\end{align*}
Now, observing that $\Omega_\lambda\to \R^N_+$ and from our hypothesis on $\rho_2$, we arrive at
$$
\rho_{p(x), \Omega}(\phi_\lambda) = \int_{\Omega} |\nabla \phi_\lambda(x)|^{p(x)}\, dx \to \int_{\R^N_+} |\nabla \phi(y)|^p\, dy\quad \text{as }\lambda\to 0+.
$$
Similar computations show that
$$
\rho_{p(x), \Omega}(\phi_\lambda) = \int_{\Omega} |\phi_\lambda(x)|^{p(x)}\, dx = O(\lambda^p),
$$
so
$$
\rho_{1,p(x), \Omega}(\phi_\lambda) = \int_\Omega |\nabla \phi_\lambda(x)|^{p(x)} + | \phi_\lambda(x)|^{p(x)}\, dx \to \int_{\R^N_+} |\nabla \phi(y)|^p\, dy\quad \text{as }\lambda\to 0+.
$$
Arguing as in the boundary term, we conclude that
$$
\lim_{\lambda\to 0+}\|\phi_\lambda\|_{1,p(x), \Omega} = \|\nabla\phi\|_{p, \R^N_+}.
$$

Now, by the definition of $T(p(\cdot),r(\cdot),\Omega,\Gamma)$, it follows that
$$
T(p(\cdot),r(\cdot),\Omega,\Gamma)\leq\frac{\|\phi_\lambda\|_{1,p(x)}}{\|\phi_\lambda\|_{r(x), \partial\Omega}}
$$
and taking the limit $\lambda\to 0+$, we obtain
$$
T(p(\cdot),q(\cdot),\Omega,\Gamma)\leq\frac{\|\nabla\phi\|_{p, \R^N_+ }}{\|\phi\|_{p_*, \partial\R^N_+}}
$$
for every $\phi\in C^\infty_c(\R^N)$. Then,
$$
T(p(\cdot),q(\cdot),\Omega,\Gamma)\leq \bar{K}(N,p)^{-1}
$$
and so, since $x_0=0$ is arbitrary,
$$
T(p(\cdot),q(\cdot),\Omega,\Gamma)\leq \inf_{x\in \A_T} \bar{K}(N,p(x))^{-1},
$$
as we wanted to show.
\end{proof}

Now we prove a Lemma that gives us some monotonicity of the constants $T(p(\cdot),q(\cdot),\Omega,\Gamma)$ with respect to $\Omega$ and $\Gamma\subset\partial\Omega$.

\begin{lema}\label{monotoniaT}
Let $\Omega_1, \Omega_2 \subset\R^N$ be two $C^2$ domains, and let $\Gamma_i\subset \partial\Omega_i$, $i=1,2$ be closed.

If $\Omega_2\subset \Omega_1$, $(\partial\Omega_2\cap \Omega_1)\subset \Gamma_2$ and $(\Gamma_1\cap\partial\Omega_2)\subset \Gamma_2$, then
$$
T(p(\cdot),q(\cdot),\Omega_1,\Gamma_1) \le T(p(\cdot),q(\cdot),\Omega_2, \Gamma_2).
$$
\end{lema}
\begin{proof}
The proof is a simple consequence that if $v\in W^{1,p(x)}_{\Gamma_2}(\Omega_2)$, then extending $v$ by $0$ to $\Omega_1\setminus \Omega_2$ gives that $v\in W^{1,p(x)}_{\Gamma_1}(\Omega_1)$.
\end{proof}

\begin{remark}
Lemma \ref{monotoniaT} will be used in the following situation: For $\Omega\subset \R^N$ and $\Gamma\subset\partial\Omega$ closed, we take $x_0\in \partial\Omega\setminus\Gamma$ and $r>0$ such that $(B_r(x_0)\cap\partial\Omega)\cap \Gamma = \emptyset$.

Then, if we call $\Omega_r := \Omega\cap B_r(x_0)$, $\Gamma_r = \partial B_r(x_0)\cap \bar \Omega$, we obtain
$$
T(p(\cdot),q(\cdot),\Omega,\Gamma) \le T(p(\cdot),q(\cdot),\Omega_r, \Gamma_r).
$$
\end{remark}

As a consequence of Lemma \ref{desigualdad.primera.T} and Lemma \ref{monotoniaT} we easily obtain the following Theorem.
\begin{teo}\label{estimacionT}
Let $\Omega\subset \R^N$ be a bounded $C^2$ domain and $\Gamma\subset\partial\Omega$ be closed. Let $p\in \PP(\Omega)$ and $r\in \PP(\partial\Omega)$ be continuous functions with modulus of continuity $\rho$ such that
$$
\ln(\lambda)\rho(\lambda)\to 0 \quad \text{as } \lambda\to 0+.
$$
Then, it holds that
$$
T(p(\cdot), r(\cdot), \Omega, \Gamma)\le \bar T \le \inf_{x\in \A_T} \bar K(N,p(x))^{-1}.
$$
\end{teo}

Now, in the spirit of \cite{FBSS1}, we use the convexity method of \cite{LPT} to prove that a minimizing sequence either is strongly convergent or concentrates around a single point.

\begin{teo}\label{alternativaT}
Assume that $r^->p^+$. Let $\{u_n\}_{n\in \N}\subset W^{1,p(x)}_\Gamma(\Omega)$ be a minimizing sequence for $T(p(\cdot), r(\cdot), \Omega, \Gamma)$. Then the following alternative holds:
\begin{itemize}
\item $\{u_n\}_{n\in \N}$ has a strongly convergence subsequence in $L^{r(x)}(\partial\Omega)$ or

\item $\{u_n\}_{n\in\N}$ has a subsequence such that $|u_n|^{r(x)}\, dS\rightharpoonup \delta_{x_0}$ weakly in the sense of measures and $|\nabla u_n|^{p(x)}\, dx\rightharpoonup \bar{T}_{x_0}^{p(x_0)}\delta_{x_0}$ weakly in the sense of measures, for some $x_0\in \A_T$ and $u_n\to 0$ strongly in $L^{p(x)}(\Omega)$.
\end{itemize}
\end{teo}

\begin{proof}
Let $\{u_n\}_{n\in\N}\subset W^{1,p(x)}_\Gamma(\Omega)$ be a normalized minimizing sequence for $T(p(\cdot), r(\cdot), \Omega, \Gamma)$, i.e.
$$
T(p(\cdot), r(\cdot), \Omega, \Gamma) = \lim_{n\to\infty} \|u_n\|_{1,p(x)}\quad\text{and}\quad \|u_n\|_{r(x), \partial\Omega} = 1.
$$
For simplicity, we denote by $T=T(p(\cdot), r(\cdot), \Omega, \Gamma)$. The concentration compactness principle for the trace immersion, Theorem \ref{CCT}, together with the estimate given in Theorem \ref{estimacionT} gives
\begin{align*}
1&=\lim_{n\to\infty}\int_{\Omega}\frac{|\nabla u_n|^{p(x)}+ |u_n|^{p(x)}}{\| u_n\|_{1, p(x)}^{p(x)}}\,dx\\
&\geq\int_{\Omega}\frac{|\nabla u|^{p(x)}  + |u|^{p(x)}}{T^{p(x)}}\,dx+\sum_{i\in I} T^{-p(x_i)}\mu_{i}\\
&\geq \min\{(T^{-1}\|u\|_{1,p(x)})^{p^+},{(T^{-1}\|u\|_{1,p(x)})^{p^-}}\} + \sum_{i\in I} \bar T_{x_i}^{-p(x_i)}\mu_{i}\\
&\geq\min\{\|u\|_{r(x), \partial\Omega}^{p^+},\|u\|_{r(x), \partial\Omega}^{p^-}\}+\sum_{i\in I}\nu_i^{\frac{p(x_i)}{p_*(x_i)}}\\
&\geq \|u\|_{r(x), \partial\Omega}^{p^+} + \sum_{i\in I} \nu_i^{\frac{p(x_i)}{p_*(x_i)}}\\
&\geq \rho_{r(x), \partial\Omega}(u)^{\frac{p^+}{r^-}} + \sum_{i\in I} \nu_i^{\frac{p(x_i)}{p_*(x_i)}}.
\end{align*}

On the other hand, since $\{u_n\}_{n\in \N}$ is normalized in $L^{r(x)}(\partial\Omega)$, we get
$$
1=\int_{\partial\Omega}|u|^{r(x)}\,dS + \sum_{i\in I} \nu_i
$$
So, since $p^+<r^-$, we can conclude that either $u$ is a minimizer of the corresponding problem and the set $I$ is empty, or $v=0$ and the set $I$ constains a single point.

If the second case occur, it is easily seen that the second alternative holds.
\end{proof}

With the aid of Theorem \ref{alternativaT} we can now prove the main result of the section.

\begin{teo}\label{condicionT}
Let $\Omega$ be a bounded domain in $\R^N$ with $\partial\Omega\in C^1$. Let $\Gamma\subset \partial\Omega$ be closed. Let $p\in \PP(\Omega)$ and $r\in\PP(\partial\Omega)$ be exponents that satisfy the regularity assumptions of Theorem \ref{estimacionT}. Assume, moreover, that $p^+<r^-$.

Then, if the following strict inequality holds $T(p(\cdot), r(\cdot), \Omega, \Gamma) <  \bar T,$ the infimum \eqref{constante 2} is attained.
\end{teo}

\begin{proof}
Let $\{u_n\}_{n\in\N}\subset W^{1,p(x)}_{\Gamma}(\Omega)$ be a minimizing sequence for \eqref{constante 2} normalized in $L^{r(x)}(\partial\Omega)$.

If $\{u_n\}_{n\in\N}$ has a strongly convergence subsequence in $L^{r(x)}(\partial\Omega)$, then the result holds.

Assume that this is not the case. Then, by Theorem \ref{alternativaT}, there exists $x_0\in\A_T$ such that $|u_n|^{r(x)}\, dS\rightharpoonup \delta_{x_0}$ and $|\nabla u_n|^{p(x)}\, dx \rightharpoonup \bar{T}_{x_0}^{p(x_0)}\delta_{x_0}$ weakly in the sense of measures.

So for $\varepsilon>0$, we have,
$$
\int_\Omega \frac{|\nabla u_n|^{p(x)} + |u_n|^{p(x)}}{(\bar{T}_{x_0}-\ve)^{p(x)}}\,dx \to \frac{\bar{T}_{x_0}^{p(x_0)}}{(\bar{T}_{x_0}-\ve)^{p(x_0)}}>1.
$$
Then, there exists $n_0$ such that for all $n\geq n_0$, we know that
$$
\| u_n\|_{1,p(x)} > \bar{T}_{x_0}-\ve.
$$
Taking limit, we obtain
$$
T(p(\cdot), r(\cdot),\Omega,\Gamma) \geq \bar{T}_{x_0}-\ve.
$$
As $\varepsilon>0$ is arbitrary, the result follows.
\end{proof}

\section{Conditions for the validity of $T(p(\cdot), r(\cdot),\Omega) < \bar T$}

In this section we investigate under what conditions the strict inequality \eqref{hip} holds. We provide two types of conditions. First, by some simple rough estimates we give {\em global} conditions, that is a condition that involves some quantities measured in the whole domain $\Omega$. This condition resembles the one found in \cite{FBR}. Then we devote ourselves to the more delicate problem of finding {\em local} conditions, that is conditions that involves some quantities computed at a single point of $\partial\Omega$. These type of conditions are in the spirit of \cite{Adi, FBS, FBSS2}.

\subsection{Global conditions}

Now we want to show an example of when the condition \eqref{hip} is guaranteed. We assume that $\Gamma=\emptyset$ and using $v=1$ as a test function we can estimate
$$
T(p(\cdot), r(\cdot), \Omega) \le \frac{\|1\|_{1,p(x)}}{\|1\|_{r(x),\partial\Omega}}.
$$

It is easy to see that
$$
\|1\|_{1,p(x)} = \|1\|_{p(x)} \leq \max\left\{|\Omega|^\frac{1}{p^+},|\Omega|^\frac{1}{p^-}\right\}
$$
and
$$
\|1\|_{r(x), \partial\Omega} \geq \min\{|\partial\Omega|^\frac{1}{r^+},|\partial\Omega|^\frac{1}{r^-}\}.
$$
So, if $\Omega$ satisfies
\begin{equation}\label{Omega<T}
\frac{\max\left\{|\Omega|^\frac{1}{p^+},|\Omega|^\frac{1}{p^-}\right\}}{\min\{|\partial\Omega|^\frac{1}{r^+},|\partial\Omega|^\frac{1}{r^-}\}}< \bar{T},
\end{equation}
then by Theorem \ref{condicionT} there exists an extremal for $T(p(\cdot), r(\cdot), \Omega)$.

Observe that the family of sets that verify \eqref{Omega<T} is large. In fact, for any open set $\Omega$ with $C^1$ boundary, if we denote $\Omega_t = t\cdot \Omega$ we have
\begin{align*}
\frac{\max\left\{|\Omega_t|^\frac{1}{p^+},|\Omega_t|^\frac{1}{p^-}\right\}}{\min\{|\partial\Omega_t|^\frac{1}{r^+},|\partial\Omega_t|^\frac{1}{r^-}\}}
 & \le \frac{t^\frac{N}{p^+}|\Omega|^\frac{1}{p^+}}{t^\frac{N-1}{r^-}|\partial \Omega|^\frac{1}{r^-}} \qquad \text{ for } t<1. 
\end{align*}
Now, the hypothesis $\frac{p^+}{r^-}<1$ imply that $\frac{N}{p^+}-\frac{N-1}{r^-}\geq \frac{1}{p^+}>0$, so we can conclude that:
$$
T(p(\cdot), r(\cdot), \Omega_t) < \bar{T},
$$
if $t>0$ is small enough.

\subsection{Local conditions}

As we mentioned in the introduction, the strategy to find local conditions for \eqref{hip} to hold is to construct a test function to estimate $T(p(\cdot), r(\cdot), \Omega, \Gamma)$ by scaling and truncating an extremal for $\bar K(N,p(x))^{-1}$ with $x\in \A_T$. 
In order for this argument to work, a result stating the equality $\bar T_x = \bar K(N,p(x))^{-1}$ is needed. This is the content of our next result.

We begin with a Lemma that is a refinement of the asymptotic expansions found in the proof of Lemma \ref{desigualdad.primera.T} since we obtain uniform convergence for bounded sets of $W^{1,p(x)}(\Omega)$. Though this lemma can be proved for variable exponents, we choose to prove it in the constant exponent case since this will be enough for our purposes and simplifies the arguments.

In order to prove the Lemma, we use the so-called {\em Fermi coordinates} in a neighborhood of some point $x_0\in\partial\Omega$. Roughly speaking the Fermi coordinates around $x_0\in \partial\Omega$ describe $x\in \Omega$ by $(y,t)$ where $y\in \R^{N-1}$ are the coordinates in a local chart of $\partial\Omega$ at $x_0$, and $t>0$ is the distance to the boundary along the inward unit normal vector.

\begin{defi}[Fermi Coordinates]\label{Fermi}\index{Fermi Coordinates}
We consider the following change of variables around a point $x_0\in\partial\Omega$.

We assume that $x_0=0$ and that $\partial\Omega$ has the following representation in a neighborhood of $0$:
$$
\partial\Omega\cap V = \{ x\in V\colon x_n = \psi(x'),\ x'\in U\subset \R^{N-1}\},\ \ \Omega\cap V = \{ x\in V\colon x_n > \psi(x'),\ x'\in U\subset \R^{N-1}\}.
$$
The function $\psi\colon U\subset \R^{N-1}\to \R$ is assume to be at least of class $C^2$ and that $\psi(0)=0$, $\nabla \psi(0)=0$.

The change of variables is then defined as $\Phi\colon U\times (0,\delta)\to \Omega\cap V$
$$
\Phi(y,t) = (y,\psi(y)) + t\nu(y),
$$
where $\nu(y)$ is the unit inward normal vector, i.e.
$$
\nu(y) = \frac{(-\nabla \psi(y), 1)}{\sqrt{1+|\nabla\psi(y)|^2}}.
$$
\end{defi}

It is well known that $\Phi$ defines a smooth diffeomorphism for $\delta>0$ small enough. 

Moreover, in \cite{Escobar} it is proved the following asymptotic expansions
\begin{lema}\label{fermi.asymptotic} With the notation introduced in Definition \ref{Fermi}, the following asymptotic expansions hold
$$
J\Phi(y,t) = 1- Ht + O(t^2 + |y|^2),
$$
where $J\Phi$ is the Jacobian, and $H$ is the mean curvature of $\partial\Omega$.

Also, if we denote $v(y,t) = u(\Phi(y,t))$ the function $u$ read in Fermi coordinates, we have 
$$
|\nabla u(x)|^2 = (\p_t v)^2+\sum_{i,j=1}^N\left(\delta^{ij}+2h^{ij}t+O(t^2+|y|^2)\right)\p_{y_i}v \p_{y_j}v,
$$
where $h^{ij}$ is the second fundamental form of $\partial\Omega$.
\end{lema}

For a general construction of the Fermi coordinates in differentiable manifolds, we refer to the book \cite{tubes}.

\begin{lema}\label{uniforme}
Let $1<p<N$ be a constant exponent and let $u$ be a smooth function on $\bar\Omega$. Then, there holds
\begin{align*}
\|u\|_{p_*, B_\ve(x_0)\cap\p \Omega} &= \eps^\frac{N-1}{p_*}(1+o(1)) \|\tilde u_\ve\|_{p_*, V\cap\p\R^N_+},\\
\|u\|_{p, B_\ve(x_0)\cap \Omega} &= \eps^\frac{N}{p}(1+o(1))\|\tilde u_\eps\|_{p, V\cap \R^N_+},\\
\|\nabla u\|_{p, B_\ve(x_0)\cap \Omega} &= \eps^{\frac{N-p}{p}}(1+o(1))\|\nabla \tilde u_\eps\|_{p, V\cap \R^N_+},
\end{align*}
where $V$ is the unit ball transformed under the Fermi coordinates, $o(1)$ is uniform in $u$ for $u$ in a bounded subset of $W^{1,p}(\Omega)$, $\tilde u_\ve(y) = \tilde u(\ve y)$, and $\tilde u $ is $u$ read in Fermi coordinates.
\end{lema}

\begin{proof}
If we denote by $\Phi(y,t)$ the change of variables from Fermi coordinates to Euclidian coordinates, then, from Lemma \ref{fermi.asymptotic} we have
$$
J\Phi = 1 + O(\ve)\quad \mbox{in } B_\ve(x_0)\cap \Omega,
$$
where $J\Phi$ is the Jacobian of $\Phi$,
$$
J_{\partial\Omega}\Phi = 1 + O(\ve) \quad \mbox{in } B_\ve(x_0)\cap \partial\Omega,
$$
where $J_{\partial\Omega}\Phi$ is the tangential Jacobian of $\Phi$ and
$$
|\nabla \tilde u_\ve| = (1+O(\ve)) |\nabla u|
$$
with $O(\ve)$ uniform in $u$.

For a more comprehensive study of the Fermi coordinates see \cite{Escobar} and the book \cite{tubes}.

Now, we simply compute
\begin{equation*}
\begin{split}
\int_{B_\ve(x_0)\cap\p \Omega} |u|^{p_*}\, dS & = \int_{(\eps\cdot V)\cap\p\R^N_+} |\tilde u(y,0)|^{p_*}(1+O(\eps))\, dy  \\
& = \eps^{N-1}(1+O(\eps)) \int_{V\cap\p\R^N_+} |\tilde u_\eps(y,0)|^{p_*}\, dy.
\end{split}
\end{equation*}

In the same way,
\begin{equation*}
\begin{split}
\int_{B_\ve(x_0)\cap\p \Omega} |\nabla u|^{p}\, dx & = \int_{(\eps\cdot V)\cap\p\R^N_+} |\nabla \tilde u(y)|^{p} (1+O(\eps))\, dy  \\
& = \eps^{N-p}(1+O(\ve)) \int_{V\cap\p\R^N_+} |\nabla \tilde u_\eps(y))|^{p}\, dy
\end{split}
\end{equation*}
and
\begin{equation*}
\begin{split}
\int_{B_\ve(x_0)\cap\p \Omega} |u|^{p}\, dx & = \int_{(\eps\cdot V)\cap\p\R^N_+} |\tilde u(y)|^{p} (1+O(\eps))\, dy  \\
& = \eps^{N}(1+O(\ve)) \int_{V\cap\p\R^N_+} |\tilde u_\eps(y))|^{p}\, dy.
\end{split}
\end{equation*}
This completes the proof.
\end{proof}

Now we can prove 
\begin{teo}\label{T=K}
Let $p\in\PP(\Omega)$ and $r\in \PP(\partial\Omega)$ be as in Theorem \ref{estimacionT}. Assume that $x_0\in \A_T$ is a local minimum of $p(x)$ and a local maximum of $r(x)$. Then
$$
\bar T_{x_0} = \bar K(N,p(x_0))^{-1}.
$$
\end{teo}

\begin{proof}
From the proof of Theorem \ref{estimacionT}, it follows that $\bar T_{x_0} \le \bar K(N,p(x_0))^{-1}$.

Let us see that if $x_0$ is a local minimum of $p(x)$ and a local maximum of $r(x)$ then the reverse inequality holds. 
Let us call $p=p(x_0)$ and then $p_*=p(x_0)_*=r(x_0)$.

Since $p(x)\ge p$, by Young's inequality with $\frac{1}{p}=\frac{1}{p(x)}+\frac{1}{s(x)}$ we obtain
$$
\int_{\Omega_\ve} |u|^p + |\nabla u|^p\, dx \le \frac{p}{p^-_\ve}\int_{\Omega_\ve} |u|^{p(x)} + |\nabla u|^{p(x)}\, dx + 2\frac{p}{s^-_\ve} |B_\ve|
$$
where $p^-_\ve = \sup_{\Omega_\ve} p(x)$.

It then follows that for any $\lambda>0$, 
$$
\|\lambda^{-1}u\|^p_{1,p,\Omega_\ve} \le (1+o(1)) \rho_{1,p(x),\Omega_\ve}(\lambda^{-1}u) + O(\ve^n).
$$
So, for any $\delta>0$, taking $\lambda = \|u\|_{1,p(x),\Omega_\ve}+\delta$ we obtain
\begin{equation}\label{desig.prim}
\|u\|_{1,p,\Omega_\ve} \le \|u\|_{1,p(x),\Omega_\ve}+\delta,
\end{equation}
if $\ve$ is small, depending only on $\delta$.

Arguing in much the same way, we obtain
\begin{equation}\label{desig.segun}
\|u\|_{r(x),\p\Omega_\ve} \le \|u\|_{p_*,\p\Omega_\ve}+\delta,
\end{equation}
for $\ve$ is small, depending only on $\delta$.

Now, by \eqref{desig.prim} and \eqref{desig.segun} it follows that
$$
\bar Q(p(\cdot),r(\cdot),\Omega_\ve)(u) = \frac{\|u\|_{1,p(x),\Omega_\ve}}{\|u\|_{r(x),\p\Omega_\ve}}\ge \frac{\|u\|_{1,p,\Omega_\ve}}{\|u\|_{p_*,\p\Omega_\ve}} + O(\delta).
$$
Finally, by Lemma \ref{uniforme}, we get
$$
\bar Q(p(\cdot),r(\cdot),\Omega_\ve)(u) \ge \frac{\|\nabla \tilde u_\ve\|_{p, V\cap\R^N_+}}{\|\tilde u_\ve\|_{p_*, V\cap \partial\R^N_+}} + o(1) + O(\delta)\ge \bar K(N,p)^{-1} + o(1) + O(\delta).
$$
So, taking infimum in $u\in W^{1,p(x)}_{\Gamma_\ve}(\Omega_\ve)$, $\ve\to 0$ and $\delta\to 0$ we obtain the desired result.
\end{proof}

With the aid of Theorem \ref{T=K} we are now in position to find local conditions to ensure the validity of $T(p(\cdot), r(\cdot),\Omega, \Gamma) < \bar T$, and so the existence of an extremal for $T(p(\cdot), r(\cdot),\Omega, \Gamma)$.

We assume, to begin with, that there exists a point $x_0\in \A_T$ such that $\bar T = \bar T_{x_0}$. Moreover, this critical point $x_0$ is assume to be a local minimum of $p(x)$ and a local maximum of $q(x)$. In view of Theorem \ref{T=K}, it follows that $\bar T = \bar T_{x_0} = \bar K(N,p(x_0))^{-1}$.

The idea, then, is similar to the one used in \cite{FBSS1}. We estimate $T(p(\cdot), r(\cdot),\Omega, \Gamma)$ evaluating the corresponding Rayleigh quotient $\bar Q(p(\cdot), q(\cdot), \Omega)$ in a properly rescaled function of the extremal for $\bar K(N,p(x_0))^{-1}$.
A fine asymptotic analysis of the Rayleigh quotient with respect to the scaling parameter will yield the desired result.

Hence the main result of the section reads

\begin{teo}\label{local.conditions T}
Let $p\in \PP(\Omega)$ and $r\in\PP(\partial\Omega)$ be $C^2$ and that $p^+<r^-$. Assume that there exists $x_0\in \A_T$ such that $\bar T =\bar T_{x_0}$ and that $x_0$ is a local minimum of $p(x)$ and a local maximum of $r(x)$. Moreover, assume that either $\partial_t p(x_0)>0$ or $H(x_0)>0$.

Then the strict inequality holds
$$
T(p(\cdot), q(\cdot), \Omega,\Gamma) < \bar T
$$
and therefore, there exists an extremal for $T(p(\cdot), q(\cdot), \Omega,\Gamma)$.
\end{teo}

We now construct the test functions needed in order to estimate the Sobolev trace constant. Assume that $0\in \p\Omega$. 
We consider the test-function 
$$ v_\eps(x)=\eta(y,t)V_{\eps, 0}(y,t), \qquad \text{ with } x=\Phi(y,t), $$
where $V$ is the extremal for $\bar K(N, p(0))^{-1}$ given by (\ref{Extremal}), and $\eta\in C^\infty_c(B_{2\delta}\times [0,2\delta),[0,1])$ is a smooth cut-off function.

From now on, we assume that $p(x)\in \PP(\Omega)$, $r(x)\in \PP(\partial\Omega)$ are of class $C^2$, $0\in \p\Omega$ and we denote $p=p(0)$ and $r=r(0)$.

The key technical tools needed in the proof of Theorem \ref{local.conditions T} are the following three Lemmas that are proved in \cite{FBSS4}.
\begin{lema}\label{propIntLp} There holds
\begin{equation}\label{IntLp}
 \int_{\Omega} f(x) |v_\eps|^{p(x)} \,dx = \bar C_0\eps^p+o(\eps^p) \quad \text{ with } \quad \bar C_0 = f(0)\int_{\R^N_+} V^p\,dx.
\end{equation}
\end{lema}

\begin{lema}\label{propA} If $p<\frac{N-1}{2}$,
\begin{equation}\label{IntBoundary}
 \int_{\p \Omega} f(x) |v_\eps|^{r(x)} \,dS_x = \bar A_0 + \bar A_1 \eps^2 \ln\eps + o(\eps^2 \ln \eps)
\end{equation}
with
$$
\bar A_0=f(0)\int_{\R^{N-1}}V(y,0)^{p_*}\,dy,
$$
and
$$ \bar A_1 = -\dfrac{1}{2p_*}f(0)\Delta r(0) \int_{\R^{N-1}} |y|^2 V(y,0)^{p_*}\,dy, $$
where $\Delta r(0)=\sum_{i=1}^{N-1}\partial_{ii}(r\circ\Phi(\cdot,0))(0)$ (equivalently, as $0$ is a critical point of $r$, $\Delta r(0)$ is also the Laplacian of $r$ with respect to the induced metric of $\partial\Omega$). 
\end{lema}

\begin{lema} \label{propD.A} Assume that $p<N^2/(3N-2)$. Then
 \begin{equation*}
\begin{split}
 \int_{\Omega} f(x)|\nabla v_\eps(x)|^{p(x)}\,dx
   = \bar D_0 + \bar D_1\eps\ln\eps + \bar D_2\eps + \bar D_3(\eps\ln\eps)^2 + \bar D_4\eps^2\ln\eps + O(\eps^2),
\end{split}
\end{equation*}
with
\begin{equation*}
\begin{split}
\bar D_0  = f(0)\int_{\R^N_+} |\nabla V|^p \,dydt, \quad
\bar D_1  = -\frac{N}{p}f(0)\p_tp(0)\int_{\R^N_+} t|\nabla V|^p \,dydt,
\end{split}
\end{equation*}
and, assuming that $\p_tp(0)=0$,
\begin{equation*}
\begin{split}
\bar D_2 & = (\p_tf(0)-Hf(0))\int_{\R^N_+}t|\nabla V|^p \,dydt 
    + p\bar h f(0) \int_{\R^N_+}\frac{t|y|^2}{(1+t)^2+|y|^2}|\nabla V|^p \,dydt,  \\
\bar D_3 & = 0, \\
\bar D_4 & = -\frac{N}{2p}f(0)\p_{tt}p(0)\int_{\R^N_+}t^2|\nabla V|^p \,dydt
         - \frac{N}{2(N-1)p}f(0)\Delta_yp(0)\int_{\R^N_+}|y|^2|\nabla V|^p \,dydt,
\end{split}
\end{equation*}
where $\bar h=\frac{1}{N-1}\sum_{i=1}^{N-1}h^{ii}(0)$ and $\Delta_yp(0)=\sum_{i=1}^{N-1}\partial_{ii}(p\circ\Phi(\cdot,0))(0)$ (which can also be seen as the Laplacian of $p|_{\partial\Omega}$ at 0 for the indiced metric of $\partial\Omega$ since the all the first derivatives of $p$ at 0 vanishes by hipotesis). 
\end{lema}

As an immediate consequence of these Lemmas we get
\begin{corol}\label{asymp.normas}
Under the same hypotheses and notations as in Lemmas \ref{propIntLp}, \ref{propA} and \ref{propD.A} we have
\begin{itemize}
\item If $\partial_t p(0)>0$,
\begin{align*}
&\|u_\ve\|_{1,p(x)} = \bar D_0^{\frac{1}{p}}\left(1+\frac{\bar D_1}{p\bar D_0}\eps\ln\eps + o(\eps\ln\eps)\right)\\
&\|u_\ve\|_{r(x),\partial\Omega} = \bar A_0^\frac{1}{p_*}\left(1 + \frac{\bar A_1}{p_*\bar  A_0}\eps^2\ln\eps + o(\eps^2\ln\eps)\right)
\end{align*}
\item If $\partial_t p(0)=0$ and $H(0)>0$,
\begin{align*}
&\|u_\ve\|_{1,p(x)} = \bar D_0^{\frac{1}{p}}\left(1+\frac{\bar D_2}{p\bar D_0}\eps + o(\eps)\right)\\
&\|u_\ve\|_{r(x),\partial\Omega} = \bar A_0^\frac{1}{p_*}\left(1 + \frac{\bar A_1}{p_*\bar  A_0}\eps^2\ln\eps + o(\eps^2\ln\eps)\right)
\end{align*}
\end{itemize}
\end{corol}

Now we are in position to prove Theorem \ref{local.conditions T}.
\begin{proof}[ Proof of Theorem \ref{local.conditions T}]
The proof is an immediate consequence of Propositions \ref{propIntLp}, \ref{propA} and \ref{propD.A}. 
In fact, without loss of generality we can assume that $x_0=0$, and let $p=p(0)$. Asume first that $\partial_t p(0)>0$. Then
\begin{align*}
T(p(\cdot), q(\cdot), \Omega,\Gamma) &\le \bar Q(p(\cdot), q(\cdot), \Omega)(u_\ve) = \frac{\bar D_0^{\frac{1}{p}}\left(1+\frac{\bar D_1}{p\bar D_0}\eps\ln\eps + o(\eps\ln\eps)\right)}{\bar A_0^\frac{1}{p_*}\left(1 + \frac{\bar A_1}{p_*\bar  A_0}\eps^2\ln\eps + o(\eps^2\ln\eps)\right)}\\
&=\bar K(N,p)^{-1} \frac{1+\frac{\bar D_1}{p\bar D_0}\eps\ln\eps + o(\eps\ln\eps)}{1 + \frac{\bar A_1}{p_* \bar A_0}\eps^2\ln\eps + o(\eps^2\ln\eps)}.
\end{align*}
The proof will be finished if we show that
$$
 \frac{1+\frac{\bar D_1}{p\bar D_0}\eps\ln\eps + o(\eps\ln\eps)}{1 + \frac{\bar A_1}{p_* \bar A_0}\eps^2\ln\eps + o(\eps^2\ln\eps)}<1,
$$
or, equivalently,
$$
\frac{\bar D_1}{p\bar D_0} + o(1) < \frac{\bar A_1}{p_* \bar A_0}\eps + o(\ve).
$$
But this former inequality holds, since $\bar D_1<0$ and $\bar D_0>0$.

The case where $\partial_t p(0)=0$ and $H(0)>0$ is analogous.
\end{proof}

\section*{Acknowledgements}
This work was partially supported by Universidad de Buenos Aires under grant UBACYT 20020100100400 and by CONICET (Argentina) PIP 5478/1438.

\def\ocirc#1{\ifmmode\setbox0=\hbox{$#1$}\dimen0=\ht0 \advance\dimen0
  by1pt\rlap{\hbox to\wd0{\hss\raise\dimen0
  \hbox{\hskip.2em$\scriptscriptstyle\circ$}\hss}}#1\else {\accent"17 #1}\fi}
  \def\ocirc#1{\ifmmode\setbox0=\hbox{$#1$}\dimen0=\ht0 \advance\dimen0
  by1pt\rlap{\hbox to\wd0{\hss\raise\dimen0
  \hbox{\hskip.2em$\scriptscriptstyle\circ$}\hss}}#1\else {\accent"17 #1}\fi}
\providecommand{\bysame}{\leavevmode\hbox to3em{\hrulefill}\thinspace}
\providecommand{\MR}{\relax\ifhmode\unskip\space\fi MR }
\providecommand{\MRhref}[2]{%
  \href{http://www.ams.org/mathscinet-getitem?mr=#1}{#2}
}
\providecommand{\href}[2]{#2}
\bibliographystyle{plain}
\bibliography{biblio}

\end{document}